\documentclass[11pt]{article}
\usepackage[utf8]{inputenc}

\usepackage{amsmath,amsthm,amssymb}
\usepackage{graphicx}
\usepackage{wrapfig}
\usepackage{enumitem}
\usepackage{tikz-cd}
\usepackage{esvect}
\usepackage{bm}
\usepackage{tocloft}
\usepackage{lipsum} 
\usepackage{mathrsfs}
\usepackage{float}
\usepackage{graphicx}
\usepackage{commath}
\usepackage{mathtools}
\usepackage{ulem}
\usepackage[nottoc]{tocbibind}

\usepackage{caption}
\usepackage{subcaption}

\usetikzlibrary{backgrounds, matrix,fit,matrix,decorations.pathreplacing, calc, positioning}

\makeatletter
\newcommand\setItemnumber[1]{\setcounter{enum\romannumeral\@enumdepth}{\numexpr#1-1\relax}}
\makeatother

\newcommand{\R}{\mathbb{R}}
\newcommand{\C}{\mathbb{C}}

\newcommand{\e}{\varepsilon}

\newcommand{\beq}{\begin{equation}}
\newcommand{\eeq}{\end{equation}}
\newcommand{\bdef}{\begin{definition}}
\newcommand{\eedef}{\end{definition}}

\newtheorem{theorem}{Theorem}[section]
\newtheorem{lemma}[theorem]{Lemma}
\newtheorem{proposition}[theorem]{Proposition}

\newtheorem{corollary}[theorem]{Corollary}
\newtheorem{definition}[theorem]{Definition}
\newtheorem{claim}[theorem]{Claim}

\theoremstyle{definition}

\newtheorem{remark}[theorem]{Remark}

\makeatletter
\newtheorem*{rep@theorem}{\rep@title}
\newcommand{\newreptheorem}[2]{%
\newenvironment{rep#1}[1]{%
 \def\rep@title{#2 \ref{##1}}%
 \begin{rep@theorem}}%
 {\end{rep@theorem}}}
\makeatother

\newreptheorem{theorem}{Theorem}
\newreptheorem{lemma}{Lemma}
\newreptheorem{proposition}{Proposition}

\PassOptionsToPackage{hyphens}{url}\usepackage{hyperref} 
\hypersetup{
  colorlinks=true,
  citecolor=blue,
  linkcolor=blue, 
  urlcolor=cyan
}

\def\M{{\mathcal M}}

\newcommand{\defeq}{\vcentcolon=}

\long\def\frame#1#2#3#4{\hbox{\vbox{\hrule height#1pt
 \hbox{\vrule width#1pt\kern #2pt
 \vbox{\kern #2pt
 \vbox{\hsize #3\noindent #4}
\kern#2pt}
 \kern#2pt\vrule width #1pt}
 \hrule height0pt depth#1pt}}}

\def\M{{\mathcal M}}

\newcommand\blfootnote[1]{%
	\begingroup
	\renewcommand\thefootnote{}\footnote{#1}%
	\addtocounter{footnote}{-1}%
	\endgroup
}

\usepackage{xcolor}
\usepackage{xpatch}
\colorlet{partnumbercolour}{blue}
\makeatletter
\AtBeginDocument{
\xpatchcmd{\@part}{\partname\nobreakspace\thepart}{\textcolor{partnumbercolour}{\thepart}}{}{}
}
\makeatother

\title{A Santal\'o inequality for the $L^p$-polar body}
\author{Vlassis Mastrantonis}

\begin{document}
\let\i\undefined
\newcommand{\i}{\sqrt{-1}}
\def\LpK{K^{\circ,p}}
\def\LpS{S^{\circ,p}}
\def\Cov{\mathrm{Cov}}
\numberwithin{equation}{section}

\maketitle

\begin{center}
\textit{Dedicated to Bo Berndtsson on the occasion of his 71\textsuperscript{st} birthday}
\end{center}

\begin{abstract}
In recent work with Berndtsson and Rubinstein, a notion of $L^p$-polarity was introduced, with classical polarity recovered in the limit $p\to\infty$, and $L^1$-polarity closely related to Bergman kernels of tube domains. A Santal\'o inequality for the $L^p$-polar was proved for symmetric convex bodies. The aim of this article is to remove the symmetry assumption. Thus, an $L^p$-Santal\'o inequality holds for any convex body after translation by the $L^p$-Santal\'o point. As a corollary, this yields an optimal upper bound on Bergman kernels of tube domains. The proof is by Steiner symmetrization, but unlike the symmetric case, a careful translation of the body is required before each symmetrization. 
\end{abstract}

\blfootnote{The author is grateful to Y.A. Rubinstein for his support and guidance, and to an anonymous referee for useful comments. 
Research supported through NSF grants DMS-1906370, 2204347, a Hauptman summer research award and a graduate school summer research fellowship at the University of Maryland.}

\section{Introduction}
In this article, we prove a Santal\'o inequality for the $L^p$-polar body, recently introduced by Berndtsson, Rubinstein, and the author \cite{BMR}. For $p\in (0,\infty),$ the \textit{$L^p$-support function} \cite[Remark 36]{MR} \cite[(1.8)]{BMR},
\begin{equation}\label{hpKdef}
    h_{p,K}(y)\defeq \log\left( \int_{K}e^{p\langle x,y\rangle}\frac{\dif x}{|K|}\right)^{\frac1p}
\end{equation}
of a convex body $K$,
generalizes the classical support function 
\begin{equation}\label{hKDef}
    h_{K}(y)\defeq \sup_{x\in K}\langle x,y\rangle,
\end{equation}
which is obtained by \eqref{hpKdef} for $p=\infty$.
Indeed, $h_{p,K}$ increases in $p$ \cite[Lemma 2.2 (v)]{BMR} and converges pointwise to \eqref{hKDef}, thus $h_{\infty, K}\defeq \lim_{p\to\infty}h_{p,K}= h_K$ \cite[Corollary 2.7]{BMR}. In contrast to $h_K$, for finite $p$, $h_{p,K}$ is not homogeneous. Nonetheless, it is always convex \cite[Lemma 2.1]{BMR} thus, by a theorem of Ball \cite[Theorem 5]{ball2} (also Klartag in the non-even case \cite[Theorem 2.2]{klartag}), it induces a near norm (that is, non-negative and sub-additive, but only positively 1-homogeneous: $\|\lambda y\|= \lambda\|y\|$ for $\lambda\geq 0$)
\begin{equation}\label{KcircpNorm}
    \|y\|_{K^{\circ,p}}\defeq \left( \frac{1}{(n-1)!}\int_0^\infty r^{n-1} e^{-h_{p,K}(ry)}\dif r\right)^{-\frac1n}. 
\end{equation}
The \textit{$L^p$-polar} is the induced convex body \cite[Definition 1.1]{BMR}
\begin{equation*}
    K^{\circ,p}\defeq \{y\in \R^n: \|y\|_{K^{\circ,p}}\leq 1\}. 
\end{equation*}
By the homogeneity of $h_K$, $\|y\|_{K^{\circ,\infty}}= h_K(y)$, thus \begin{equation*}
    K^{\circ, \infty}= K^{\circ}\defeq \{y\in \R^n: \langle x,y\rangle\leq 1, \text{ for all } x\in K\}, 
\end{equation*}
coincides with the polar of $K$.

The \textit{$L^p$-Mahler volume} is defined via \cite[(1.9), Theorem 1.2]{BMR}
\begin{equation}\label{MpEq}
    \M_p(K)\defeq |K|\int_{\R^n} e^{-h_{p,K}(y)}\dif y= n! |K| |K^{\circ,p}|.
\end{equation}
Denote by 
\begin{equation*}
    B_2^n\defeq \{x\in \R^n: x_1^2+ \ldots+ x_n^2\leq 1\}.
\end{equation*}
Santal\'o's inequality \cite[(3.12)]{santalo} was generalized to the $L^p$-polar for symmetric convex bodies. 
\begin{theorem}{\upshape\cite[Theorem 1.6]{BMR}}
\label{SantaloSym}
    Let $p\in (0,\infty]$. For a symmetric convex body $K\subset \R^n$, $\M_p(K)\leq \M_p(B_2^n)$.
\end{theorem}

The main result of the present paper is to remove the symmetry assumption:
\begin{theorem}
\label{LpSantaloThm}
Let $p\in (0,\infty]$. For a convex body $K\subset \R^n$, 
\begin{equation}\label{LpSantaloIneq}
    \inf_{x\in \R^n}\M_p(K-x)\leq \M_p(B_2^n).
\end{equation}
\end{theorem}
 
Theorem \ref{LpSantaloThm} is a strengthening of Theorem \ref{SantaloSym} because the infimum on the left-hand side of \eqref{LpSantaloIneq} is attained at the origin for symmetric convex bodies. 
As we show in our previous work,
the infimum in \eqref{LpSantaloIneq} is always a minimum attained at a unique point that lies in the interior of $K$. 
\begin{definition}{\upshape\cite[Proposition 1.5]{BMR}}\label{LpSantaloPointDef}
    Let $p\in (0,\infty]$, and $K\subset \R^n$ a convex body. Denote by $s_p(K)$ the unique point in $\mathrm{int}\,K$ for which 
    \begin{equation}\label{tMpDef}
    \M_p(K-s_p(K))= \inf_{x\in \R^n} \M_p(K-x),
    \end{equation}
    called the $L^p$-Santal\'o point of $K$.
\end{definition}
The $L^p$-Santal\'o point is uniquely characterized by the vanishing of the barycenter of $h_{p,K}$.
Denote the \textit{barycenter} of a convex body $K$ by
\begin{equation*}
    b(K)\defeq \int_K x \frac{\dif x}{|K|}. 
\end{equation*}
For a convex function $\phi: \R^n\to \R\cup\{\infty\}$, denote by 
\begin{equation*}
    V(\phi)\defeq \int_{\R^n} e^{-\phi(x)}\dif x, \quad \text{ and } \quad b(\phi)\defeq \int_{\R^n}x e^{-\phi(x)}\frac{\dif x}{V(\phi)},
\end{equation*}
its volume and barycenter respectively. For the convex indicator function of $K$, $\bm{1}_K^\infty$ being $0$ for $x\in K$ and $\infty$ otherwise, $V(\bm{1}_K^\infty)= |K|$ and $b(\bm{1}_K^\infty)=b(K)$. 
\begin{lemma}{\upshape\cite[Proposition 1.5]{BMR}}
\label{LpSPbarycenter}
    Let $p\in(0,\infty]$, and $K\subset \R^n$ a convex body. For $x\in \R^n$, $x= s_p(K)$ if and only if $b(h_{p,K-x})=0$.
\end{lemma}

For $p=\infty$, Theorem \ref{LpSantaloThm} is due to Santal\'o \cite[(3.12)]{santalo} and, in fact, $s_\infty(K)$ is the Santal\'o point of $K$. The maximzers were uniquely characterized as ellipsoids (affine images of $B_2^n$) by Petty \cite[Theorem 3.2]{petty}. 
In future work we hope to characterize the equality case for finite $p$.
A proof using Steiner symmetrization was later given by Saint-Raymond for symmetric convex bodies \cite{saint-raymond}. Meyer--Pajor adapted the proof to any convex body \cite{meyer-pajor}. Our proof is mainly inspired by their work.

It is worthwhile to make a comparison between our result and the Santal\'o inequality for the $L^p$-centroid bodies due to Lutwak--Zhang \cite{LutwakZhang}. The $L^p$-centroid body of a convex body is the symmetric convex body $\Gamma_pK$ with support function 
\begin{equation*}
    h_{\Gamma_pK}(y)\defeq \left( \int_{K}|\langle x,y\rangle|^p \frac{dx}{|K|} \right)^{\frac1p}. 
\end{equation*}
Lutwak--Zhang proved that, among convex bodies, the product $|K||(\Gamma_pK)^\circ|$ is maximized by ellipsoids centered at the origin \cite[Theorem B]{LutwakZhang}. 
Their proof follows a similar line of reasoning, namely that Steiner symmetrization increases the volume product \cite[Lemma 3.2]{LutwakZhang}.
Nonetheless, we are not aware of a direct relation between the $L^p$-polar $K^{\circ,p}$ and the $L^p$-centroid $\Gamma_pK$, despite the fundamental differences \cite[Remark 1.7]{BMR}. An important hurdle in the analysis of the $L^p$-polar is that 
we do not know of a useful description of its support function; the $L^p$-support function, even though the defining function, is seemingly unrelated to its support function. To surpass this hurdle, we rely on a theorem of Ball (Theorem \ref{BallIneq}) connecting the $L^p$-support to the near-norm of the $L^p$-polar. 

Finally, Theorem \ref{LpSantaloThm} can be used to obtain an upper bound on the Bergman kernels of tube domains evaluated on the diagonal. The connection between $\M$ and Bergman kernels goes back to Nazarov \cite{nazarov}, B\l{}ocki \cite{blocki2}, and Berndtsson \cite{berndtsson2} who used Bergman kernels to obtain lower bounds on $\M$. 
For a convex body $K\subset \R^n$, let
\begin{equation*}
    T_K\defeq \R^n+ \i \,\mathrm{int}\, K, 
\end{equation*}
and
\begin{equation*}
    A^2(T_K)\defeq \{f: T_K\to \C: \text{ holomorphic with } \int_{T_K}|f(z)|^2\dif\lambda(z)<\infty\},
\end{equation*}
be the Bergman space of $T_K$, 
where $\lambda$ denotes the Lebesgue measure. 
As observed in our previous work,
$\M_1$ and the Bergman kernel of $A^2(T_K)$ evaluated at the diagonal are related by the following \cite[(42)]{MR}, 
\begin{equation*}
    \M_1(K-x)= (4\pi)^n |K|^2 \sup_{\substack{f\in A^2(T_{K})\\ f\not\equiv 0}} \frac{|f(\i x)|^2}{\|f\|^2_{L^2(T_K)}}, \quad \text{ for all } x\in \mathrm{int}\, K. 
\end{equation*}
Therefore, by Theorem \ref{LpSantaloThm} for $p=1$, we obtain the following estimate. 
\begin{corollary}
Let $K\subset \R^n$ be a convex body. For $f\in A^2(T_K)$,
\begin{equation*}
    |f(\i s_1(K))|^2\leq \frac{\M_1(B_2^n)}{(4\pi)^n |K|^2} \int_{T_K}|f(z)|^2\dif\lambda(z).
\end{equation*}
\end{corollary}

\bigskip
\noindent
\textbf{Organization}
In \S \ref{proofOutline}, an outline of the proof of Theorem \ref{LpSantaloThm} is given, stating without proof a proposition and three lemmas necessary for its proof. In \S \ref{steinerSection}, a few facts about Steiner symmetrization are recalled. \S \ref{pSantaloSection} is dedicated to demonstrating that the $L^p$-Santal\'o point (Definition \ref{LpSantaloPointDef}) of a convex body symmetric with respect to hyperplane lies on the hyperplane (Lemma \ref{LpSymSP}). In \S \ref{separatingSection}, starting with a convex body $K$ that contains the origin in the interior, it is shown that it can be translated in any direction $u$ so that $u^\perp$ $\lambda$-separates (Definition \ref{separatingDef}) the $L^p$-polar (Lemma \ref{SeparatingLemma}).
A necessary generalization of a theorem of Ball is proved in \S \ref{BallSection}. 
In \S \ref{SteinerMonotonicitySection}, it is shown that starting with a convex body $K$ and a direction $u$ so that $u^\perp$ $1/2$-separates $K^{\circ,p}$, then Steiner symmetrization increases the $L^p$-Mahler volume (Lemma \ref{VolumeLemma}). An essential step in the proof of Lemma \ref{VolumeLemma}, an inequality between the $L^p$-support functions of $K$ and its Steiner symmetrization (Lemma \ref{technicalLemma}), is detailed in \S \ref{SupportIneqSection}. Finally, Theorem \ref{LpSantaloThm} is proved in \S \ref{FinishingSection}.

\section{Inequality}
\label{inequalitySection}
\subsection{Outline of the proof of Theorem \ref{LpSantaloThm}}\label{proofOutline}

The proof of Theorem \ref{LpSantaloThm} is by Steiner symmetrization. For a unit vector $u\in\partial B_2^n$, denote by
\begin{equation*}
    u^\perp\defeq \{x\in\R^n: \langle x,u\rangle=0\},
\end{equation*}
the hyperplane through the origin perpendicular to $u$.
Let, also, 
\begin{equation*}
    \pi_{u^\perp}: \R^n\ni x\mapsto x-\langle x,u\rangle u\in u^\perp, 
\end{equation*}
the projection onto $u^\perp$. 
\begin{definition}
For a convex body $K\subset \R^n$ and $u\in\partial B_2^n$, the Steiner symmetral of $K$ in the $u$ direction is given by 
\begin{equation*}
    \sigma_uK\defeq \{x+tu\in \R^n: x\in \pi_{u^\perp}(K)\text{ and } |t|\leq \frac12 |K\cap (x+\R u)|\}. 
\end{equation*}
\end{definition}
We know that Steiner symmetrization always increases the $L^p$-Mahler volume of a symmetric convex body \cite[Proposition 5.1]{BMR}. This is not true without the symmetry assumption. Indeed, recall the following. 
\begin{lemma}{\upshape\cite[Lemma 4.2]{BMR}}
\label{finitenessLemma}
    Let $p\in (0,\infty]$. For a convex body $K\subset \R^n$, $\M_p(K)$ is finite if and only if $0\in\mathrm{int}\,K$. 
\end{lemma}
Therefore,
given an appropriately misplaced convex body $K\subset \R^n$, say with $0\not\in\mathrm{int}\,K$, and $u\in\partial B_2^n$ such that $0\in\mathrm{int}\,\sigma_uK$, then
\begin{equation*}
    \infty= \M_p(K)> \M_p(\sigma_uK).
\end{equation*}

\begin{figure}[H]
    \centering
\begin{tikzpicture}[scale=1]
\draw[thick] (0,2) -- (2,1) -- (-1,1) -- cycle;
\draw[thick] (-1,0) -- (0,.5) -- (2,0) -- (0,-.5) -- cycle;
\draw[->] (-2.5,0) -- (3,0) node[anchor=north west] {x};
\draw[->] (0,-1) -- (0,2.5) node[anchor=north west] {y};
\draw[dotted, thick] (-1,1) -- (-1,0);
\draw[dotted, thick] (0,2) -- (0,-1/2);
\draw[dotted, thick] (2,1) -- (2,0);
\node[above right=1pt] at (1, 1.5){$K$};
\node[below right=1pt] at (1, -.3){$\sigma_{e_2}K$};
\end{tikzpicture}  
    \caption{Steiner symmetrization could potentially decrease $L^p$-Mahler volume.}
    \label{SteinerMahlerDecreaseFig}
\end{figure}
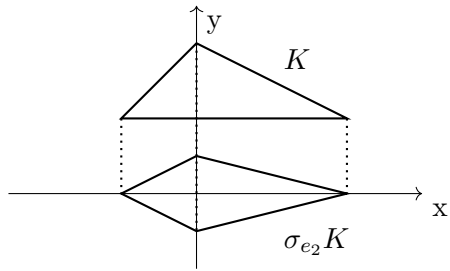

It is, thus, necessary to reposition $K$ before applying Steiner symmetrization. 
Our proof is inspired by the works of Meyer--Pajor \cite{meyer-pajor} and Meyer--Reisner \cite{meyer-reisner}. However, the construction in the proof of their main theorem of Meyer--Pajor \cite[Theorem]{meyer-pajor} cannot be entirely generalized in our setting; duality of the polar body is heavily used, but for finite $p$, $L^p$-polarity is no longer an involution on convex bodies.

Even though Steiner symmetrization does not always increase $\M_p$ (Figure \ref{SteinerMahlerDecreaseFig}), it, nonetheless, increases \eqref{tMpDef}. 
\begin{proposition}
\label{SteinerProp}
Let $p\in (0,\infty]$. For a convex body $K\subset \R^n$ and $u\in\partial B_2^n$, 
$$
    \inf_{x\in \R^n}\M_p(K-x)\leq \inf_{x\in \R^n}\M_p((\sigma_u K)-x).
$$
\end{proposition}

The main ingredient for the proof of Proposition \ref{SteinerProp} is a generalization of a lemma of Meyer--Pajor \cite[Lemma 7]{meyer-pajor} to the $L^p$-polar (Lemma \ref{VolumeLemma} below). First, a couple of definitions. For a unit vector $u\in \partial B_2^n$, let
\begin{equation*}
    \begin{gathered}
        (u^\perp)^+\defeq \{x\in \R^n: \langle x,u\rangle\geq 0\}, \\
        (u^\perp)^-\defeq \{x\in \R^n: \langle x,u\rangle\leq 0\}, 
    \end{gathered}
\end{equation*}
the two closed half-spaces defined by $u^\perp$.

\begin{definition}\label{separatingDef}
    Let $K\subset \R^n$ be a convex body, $u\in \partial B_2^n$, and $\lambda\in (0,1)$. The hyperplane $u^\perp$ $\lambda$-separates $K$ if 
    \begin{equation*}
        \frac{|K\cap (u^\perp)^+|}{|K|} \frac{|K\cap (u^\perp)^-|}{|K|}= \lambda(1-\lambda). 
    \end{equation*}
\end{definition}

\begin{remark}
    Definition \ref{separatingDef} is equivalent to having either $|K\cap (u^\perp)^+|= \lambda |K|$ and $|K\cap (u^\perp)^-|= (1-\lambda)|K|$, or $|K\cap (u^\perp)^+|= (1-\lambda) |K|$ and $|K\cap (u^\perp)^-|= \lambda|K|$. To see why, it is enough to note that for $x, \lambda\in (0,1)$, $x(1-x)= \lambda(1-\lambda)$ if and only if $x= \lambda$ or $x=1-\lambda$.
\end{remark}

\begin{lemma}
\label{VolumeLemma}
Let $p\in (0,\infty]$.
For $u\in\partial B_2^n$ and $\lambda\in (0,1)$ such that $u^\perp$ $\lambda$-separates $K^{\circ,p}$, 
\begin{equation*}
    |(\sigma_u K)^{\circ, p}|\geq 4\lambda(1-\lambda) |K^{\circ,p}|.
\end{equation*}
\end{lemma}

The proof of Lemma \ref{VolumeLemma} is rather technical and so its proof is delayed until \S \ref{SteinerMonotonicitySection}. 
Given Lemma \ref{VolumeLemma}, to prove Proposition \ref{SteinerProp} we capitalize  on the following two observations.
First, the Steiner symmetral of a convex body $K$ translated by a vector $x$ perpendicular to the direction of the symmetrization $u$, is the Steiner symmetral of $K$ with respect to $u$ translated by $x$,
\begin{equation}\label{perp}
    \sigma_u(K-x)= (\sigma_uK) -x, \quad \text{ for all } x\in u^\perp. 
\end{equation}
Indeed, it is enough to note that since $x\in u^\perp$, $\pi_{u^\perp}(K-x)= \pi_{u^\perp}(K)-x$. Therefore, $z+ tu\in \sigma_u(K-x)$ if and only if $z\in \pi_{u^\perp}(K-x)= \pi_{u^\perp}(K)-x$ and $2|t|\leq |(K-x)\cap (z+\R u)|$. Or, equivalently, $z+x\in \pi_{u^\perp}(K)$ and $2|t|\leq |K\cap (z+x+ \R u)|$, because $(K-x)\cap (z+\R u)= \big(K\cap (z+x+ \R u)\big)-x$, i.e., $z+ x+ tu\in \sigma_u K$ proving \eqref{perp}. 
\begin{figure}[H]
    \centering
\begin{tikzpicture}[scale=1]
\draw[thick, dashed] (0,2) -- (2,1) -- (-1,1) -- cycle;
\draw[thick] (2,2) -- (4,1) -- (1,1) -- cycle;
\draw[thick, dashed] (-1,0) -- (0,.5) -- (2,0) -- (0,-.5) -- cycle;
\draw[thick] (1,0) -- (2,.5) -- (4,0) -- (2,-.5) -- cycle;
\draw[->] (-2.5,0) -- (4.5,0) node[anchor=north west] {x};
\draw[->] (0,-1) -- (0,2.5) node[anchor=north west] {y};
\node[above left=1pt] at (-.5, 1.5){$K$};
\node[below left=1pt] at (-.5, -.3){$\sigma_{e_2}K$};
\node[above right=1pt] at (2.8, 1.5){$K-2e_1$};
\node[below right=1pt] at (2.5, -.3){$\sigma_{e_2}(K-2e_1)$};
\end{tikzpicture}  
    \caption{Translating $K$ in a direction normal to $u$ has the same effect on the Steiner symmetral.}
    \label{fig:SteinerPerp}
\end{figure}
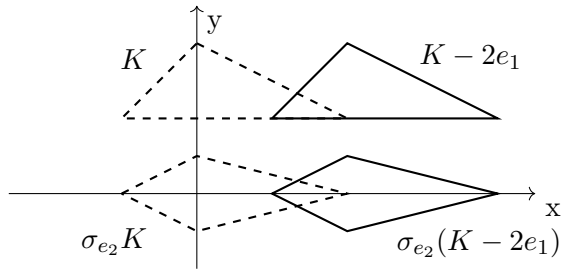

Second, the Steiner symmetral of a convex body $K$ with respect to a unit vector $u$ remains invariant under translation of the body in the direction of $u$, i.e., 
\begin{equation}\label{parallel}
    \sigma_u(K-tu)= \sigma_uK, \quad \text{ for all } t\in \R. 
\end{equation}
This follows directly from $\pi_{u^\perp}(K-tu)= \pi_{u^\perp}(K)$ and $(K-tu)\cap (x+\R u) = \big(K\cap (x+\R u)\big) - tu$. 
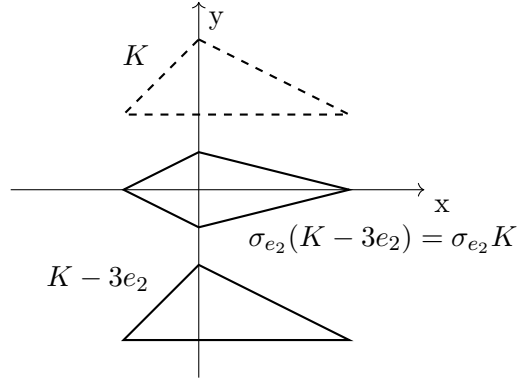
\begin{figure}[H]
    \centering
\begin{tikzpicture}[scale=1]
\draw[thick, dashed] (0,2) -- (2,1) -- (-1,1) -- cycle;
\draw[thick] (0,-1) -- (2,-2) -- (-1,-2) -- cycle;
\draw[thick] (-1,0) -- (0,.5) -- (2,0) -- (0,-.5) -- cycle;
\draw[->] (-2.5,0) -- (3,0) node[anchor=north west] {x};
\draw[->] (0,-2.5) -- (0,2.5) node[anchor=north west] {y};
\node[above left=1pt] at (-.5, 1.5){$K$};
\node[below right=1pt] at (.5, -.3){$\sigma_{e_2}(K-3e_2)=\sigma_{e_2}K$};
\node[above left=1pt] at (-.5, -1.5){$K-3e_2$};
\end{tikzpicture}  
    \caption{Translating $K$ in a direction parallel to $u$ has no effect on the Steiner symmetral.}
    \label{fig:SteinerParallel}
\end{figure}

The first observation \eqref{perp}, will be used to place $K$ in a position such that the $L^p$-Santal\'o point of the Steiner symmetral lies at the origin, $s_p(\sigma_uK)=0$, by translating $K$ in a direction perpendicular to $u$. This is always possible because the $L^p$-Santal\'o point of $\sigma_uK$ lies on $u^\perp$. The proof occupies \S \ref{pSantaloSection}.

\begin{lemma}
\label{SantaloPointLemma}
    Let $p\in(0,\infty]$. For a convex body $K\subset \R^n$ and $u\in\partial B_2^n$, $s_p(\sigma_uK)\in u^\perp$.
\end{lemma}

The second observation \eqref{parallel} is used to further translate $K$ in a direction parallel to $u$, ensuring that the $L^p$-Santal\'o point of the Steiner symmetral remains at the origin, so that, in addition, $u^\perp$ $1/2$-separates $K^{\circ,p}$. Consequently, by Lemma \ref{VolumeLemma}, Steiner symmetrization can be used to increase $\M_p$.

\begin{lemma}
\label{SeparatingLemma}
Let $p\in (0,\infty]$, and $K\subset \R^n$ a convex body with $0\in\mathrm{int}\,K$. For $u\in \partial B_2^n$ and $\lambda\in (0,1)$, there is $t\in \R$ such that $u^\perp$ $\lambda$-separates $(K-tu)^{\circ,p}$. 
\end{lemma}

\begin{figure}[H]
    \centering
\begin{tikzpicture}[baseline, scale=1]
\draw[thick, dashed] (0,2) -- (2,1) -- (-1,1) -- cycle;
\draw[thick] (0,.7) -- (2,-.3) -- (-1,-.3) -- cycle;
\draw[->] (-2.5,0) -- (3,0) node[anchor=north west] {x};
\draw[->] (0,-1) -- (0,2.5) node[anchor=north west] {y};
\node[above right=1pt] at (1, 1.5){$K$};
\node[below right=1pt] at (1, -0.5){$K- te_2$};
\end{tikzpicture}  
\begin{tikzpicture}[baseline, scale=.8]
\draw[thick] (-10/7, 10/7) -- (0,-10/3) -- (5/7, 10/7) -- cycle;
\draw[->] (-2.5,0) -- (3,0) node[anchor=north west] {x};
\draw[->] (0,-4) -- (0,3) node[anchor=north west] {y};
\node[above right=1pt] at (.8, .5){$(K-te_2)^\circ\cap (e_2^\perp)^+$};
\node[above right=1pt] at (.5, -2){$(K-te_2)^\circ\cap (e_2^\perp)^-$};
\end{tikzpicture}  
    \caption{Translating $K$ in a direction parallel to $u$ so that $u^\perp$ $1/2$-separates the $L^p$-polar.}
\end{figure}
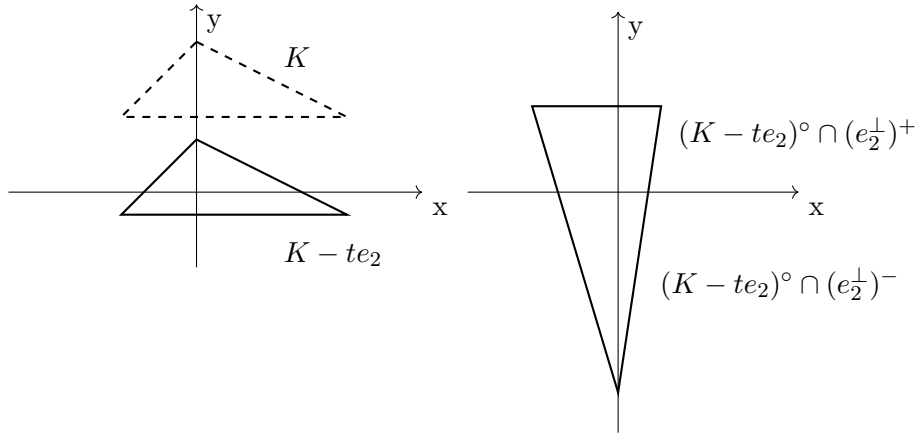

Given a direction $u\in \partial B_2^n$,
Lemmas \ref{VolumeLemma}, \ref{SantaloPointLemma}, and \ref{SeparatingLemma}, allow us to reposition $K$ in a way that achieves two objectives simultaneously: first, $u^\perp$ $1/2$-separates its $L^p$-polar, resulting in an increased $L^p$-Mahler volume of the Steiner symmetral, and second, the $L^p$-Mahler volume of its Steiner symmetral attains its minimum value over translations. This is sufficient to prove Proposition \ref{SteinerProp}. Theorem \ref{LpSantaloThm} then follows by repeated applications of Steiner symmetrization.

\subsection{Steiner symmetrization}
\label{steinerSection}

We recall a few facts about Steiner symmetrization \cite[pp. 286--287]{steiner} without proof. 
\begin{definition}
Let $u\in \partial B_2^n$.
A convex body $K\subset \R^n$ is symmetric with respect to a hyperplane $u^\perp$ if for all $x\in K$,
$
    x- 2\langle x,u\rangle u \in K.
$
\end{definition}
Steiner symmetrization produces a convex body that is symmetric with respect to $u^\perp$. 
It also preserves convexity and volume \cite[Proposition 9.1]{gruber}. 
\begin{lemma}\label{SteinerProperties}
    For a convex body $K\subset \R^n$ and $u\in\partial B_2^n$, $\sigma_u K$ is a convex body, symmetric with respect to $u^\perp$, with $|\sigma_u K|= |K|$.
\end{lemma}

Since $L^p$-polarity preserves the symmetries of the convex body with respect to hyperplanes \cite[Lemma 5.17]{BMR}, by Lemma \ref{SteinerProperties}, $(\sigma_u K)^{\circ,p}$ is symmetric with respect to $u^\perp$. 
\begin{lemma}\label{LpPolarSteinerSym}
    Let $p\in (0,\infty]$. For a convex body $K\subset \R^n$ and $u\in \partial B_2^n$, $(\sigma_u K)^{\circ,p}$ is symmetric with respect to $u^\perp$. 
\end{lemma}

For a matrix $A\in GL(n,\R)$, $A^T$ denotes the transpose of $A$, and $O(n)\defeq \{A\in GL(n,\R): A^T= A^{-1}\}$ the orthogonal group.
The next lemma allows for working with $e_n$ instead of an arbitrary unit vector \cite[Lemma 5.8]{BMR}. 
\begin{lemma}\label{steiner_orthogonal}
For a convex body $K\subset \R^n$, $u\in \partial B_2^n$, and $A\in O(n)$, 
\begin{equation*}
    \sigma_{u}(K)= A^{-1}\sigma_{Au}(AK). 
\end{equation*}
In particular, $|\sigma_u(K)|= |\sigma_{Au}(AK)|$.
\end{lemma}

Finally, Steiner symmetrization preserves symmetries in orthogonal directions \cite[Lemma 3.1]{xi-zhao}. 
\begin{lemma}\label{SteinerOrthogonal}
    Let $K\subset \R^n$ be a convex body, and $u\in \partial B_2^n$ such that $K$ is symmetric with respect to $u^\perp$. For $v\in \partial B_2^n\cap u^\perp$, $\sigma_vK$ is symmetric with respect to $u^\perp$. 
\end{lemma}

\subsection{\texorpdfstring{Symmetries of the $L^p$-Santal\'o point}{Symmetries of the Lᵖ-Santaló point}}
\label{pSantaloSection}
In this section, we prove Lemma \ref{SantaloPointLemma}. 

A measure $\mu$ on $\R^n$ is called \textit{absolutely continuous} with respect to another measure $\nu$, if $\mu(A)=0$ for all measurable $A\subset \R^n$ with $\nu(A)=0$. As a consequence, if for a positive function $f: \R^n\to [0,\infty)$, 
\begin{equation*}
    \int_{\R^n}f(x)\dif\mu(x)=0, 
\end{equation*}
then $f$ is zero $\mu$-almost everywhere, and hence $f$ is zero $\nu$-almost everywhere.

\begin{lemma}\label{ACLemma}
    Let $p\in (0,\infty]$. For a convex body $K\subset \R^n$, the Lebsegue measure $\lambda$ is absolutely continuous with respect to the measure $e^{-h_{p,K}(y)}\dif y$.
\end{lemma}
\begin{proof}
    Let us start with $p=\infty$. Let $A\subset \R^n$ be a measurable set such that $\int_{A} e^{-h_K(y)}\dif y=0$. For $r>0$, 
    \begin{equation*}
        0= \int_{A} e^{-h_K(y)}\dif y\geq \int_{A\cap \{h_K\leq r\}} e^{-h_K(y)}\dif y\geq \lambda(A\cap (rK^\circ)) e^{-r}\geq 0,
    \end{equation*}
    because $\{h_K\leq r\}= rK^\circ$, by the homogeneity of $h_K$. Therefore, $\lambda(A\cap (rK^\circ))=0$ for all $r>0$. Taking $r\to \infty$, $\lambda(A)=0$.

    For $p\in (0,\infty)$, note
    \begin{equation*}
        h_{p,K}(y)= \frac1p\log\int_K e^{p\langle x,y\rangle}\frac{\dif x}{|K|}\leq \frac1p \log\int_K e^{p h_K(y)}\frac{\dif x}{|K|}= h_K(y). 
    \end{equation*}
    Therefore, for measurable $A\subset \R^n$ with $\int_A e^{-h_{p,K}(y)}\dif y=0$, 
    \begin{equation*}
        0= \int_{A} e^{-h_{p,K}(y)}\dif y\geq \int_A e^{-h_{K}(y)}\dif y\geq 0, 
    \end{equation*}
    and hence $\int_{A} e^{-h_K(y)}\dif y=0$. By the $p=\infty$ case, $\lambda(A)=0$.
\end{proof}

The $L^p$-support function inherits symmetries with respect to hyperplanes. 
\begin{claim}
    Let $p\in (0,\infty]$, $K\subset \R^n$ a convex body, and $u\in \R^n$. If $K$ is symmetric with respect to $u^\perp$, then $h_{p,K}$ is symmetric with respect to $u^\perp$, i.e., 
    \begin{equation*}
        h_{p,K}(y+su)= h_{p,K}(y-su), \quad \text{ for all } y\in u^\perp, s\in \R. 
    \end{equation*}
\end{claim}
\begin{proof}
Since $K$ is symmetric with respect to $u^\perp$, there is $f: K\cap u^\perp\to [0,\infty)$ such that
\begin{equation*}
    K= \{x+tu: x\in K\cap u^\perp, |t|\leq f(x)\}. 
\end{equation*}
For $y\in u^\perp, s\in \R$, 
    \begin{equation*}
        \begin{aligned}
            h_{p,K}(y+su)&= \frac1p \log\int_{K\cap u^\perp} \int_{-f(x)}^{f(x)} e^{p\langle x+tu, y+su\rangle}\frac{\dif t\dif x}{|K|}\\
            &= \frac1p \log\int_{K\cap u^\perp}\int_{-f(x)}^{f(x)} e^{p\langle x,y\rangle} e^{pts}\frac{\dif t\dif x}{|K|}\\
            &= \frac1p\log\int_{K\cap u^\perp}\int_{-f(x)}^{f(x)} e^{p\langle x,y\rangle} e^{-p\tau s}\frac{\dif \tau\dif x}{|K|}\\
            &= \frac1p \log\int_{K\cap u^\perp} \int_{-f(x)}^{f(x)} e^{p\langle x+ \tau u, y- su\rangle}\frac{\dif \tau\dif x}{|K|} \\
            &= h_{p,K}(y-su), 
        \end{aligned}
    \end{equation*}
    by the change of variables $\tau= -t$.
\end{proof}

The $L^p$-polar behaves exactly like the polar under the action of $GL(n,\R)$ \cite[(4.7)]{BMR}, that is for a convex body $K\subset \R^n$ and $A\in GL(n,\R)$, 
\begin{equation}\label{KcircpGLnR}
    (AK)^{\circ,p}= (A^{-1})^T K^{\circ,p}.
\end{equation}
To see why, by a simple change of variables, $h_{p,AK}(y)= h_{p,K}(A^T y)$ \cite[Lemma 2.2 (iii)]{BMR}. Therefore, 
\begin{equation*}
\begin{aligned}
    \|y\|_{(AK)^{\circ,p}}\defeq &\left( \int_0^\infty r^{n-1}e^{-h_{p,AK}(ry)}\dif r\right)^{-\frac{1}{n}}\\
    = &\left( \int_0^\infty r^{n-1}e^{-h_{p,K}(rA^T y)}\dif r\right)^{-\frac{1}{n}}= \|A^T y\|_{K^{\circ,p}}, 
\end{aligned}
\end{equation*}
from which \eqref{KcircpGLnR} follows.

\begin{lemma}\label{LpSantaloGL}
    Let $p\in (0,\infty]$. For a convex body $K\subset \R^n$, $A\in GL(n,\R)$, $s_p(AK)= A s_p(K)$. 
\end{lemma}
\begin{proof}
By Lemma \ref{LpSPbarycenter},
    the vanishing of the barycenter of the $L^p$-support function uniquely characterizes the $L^p$-Santal\'o point. 
    By changing variables, 
    \begin{equation*}
        h_{p, AK-As_p(K)}(y)= h_{p, A(K-s_p(K))}(y)= h_{p, K-s_p(K)}(A^T y), 
    \end{equation*}
    \cite[Lemma 2.2 (iii)]{BMR}, and
   $V(h_{p, AK-As_p(K)})= |\det A|^{-1} V(h_{p, K-s_p(K)})$. 
    For the barycenter, 
    \begin{equation*}
        \begin{aligned}
            b(h_{p, AK-As_p(K)})&= \int_{\R^n} y e^{-h_{p,AK-As_p(K)}(y)}\frac{\dif y}{V(h_{p, AK-As_p(K)})} \\
            &= \int_{\R^n} y e^{-h_{p,K-s_p(K)}(A^T y)} \frac{\dif y}{|\det A|^{-1} V(h_{p,K-s_p(K)})} \\
            &= (A^T)^{-1} \int_{\R^n} z e^{-h_{p,K-s_p(K)}(z)}\frac{|\det A|^{-1}\dif z}{|\det A|^{-1} V(h_{p, K-s_p(K)})} \\
            &= (A^T)^{-1} b(h_{p,K-s_p(K)}). 
        \end{aligned}
    \end{equation*}
    Therefore, $b(h_{p,K-s_p(K)})=0$ if and only if $b(h_{p, AK-As_p(K)})=0$, from which the claim follows.
\end{proof}

For Lemma \ref{SantaloPointLemma}, let us prove something slightly more general. 
\begin{lemma}\label{LpSymSP}
    Let $p\in (0,\infty]$, $K\subset \R^n$ a convex body, and $u\in \partial B_2^n$. If $K$ is symmetric with respect to $u^\perp$, then $s_p(K)\in u^\perp$. 
\end{lemma}
\begin{proof}
    To begin with, take $u= e_n$.
    Since $K\subset \R^n$ is a convex body symmetric with respect to $e_n^\perp$ there is $f: K\cap e_n^\perp \to [0,\infty)$ such that
    \begin{equation*}
        K= \{(\xi, x_n) \in\R^{n-1}\times \R: (\xi,0)\in K\cap u^\perp \text{ and } |x_n|\leq f(x)\}.
    \end{equation*}
    Let
    \begin{equation*}
        s_p(K)= (\xi_0, t_0), \quad (\xi_0,0)\in K\cap u^\perp, t_0 \in \R,
    \end{equation*}
    be the $L^p$-Santal\'o point of $K$. The claim is that $t_0=0$. Indeed,  by Lemma \ref{LpSPbarycenter}, the barycenter of $h_{p,K-s_p(K)}$ lies at the origin. In particular, the $n$-th coordinate is 0,
        {\allowdisplaybreaks
    \begin{align*}
            0&= \int_{\R^n} y_n e^{-h_{p,K-s_p(K)}(y)}\dif y\\
            &= \int_{\R^{n-1}} \int_\R  y_n e^{-h_{p,K}(\eta, y_n)} e^{\langle (\xi_0, t_0), (\eta, y_n)\rangle}\dif y_n\dif \eta \\
            &= \int_{\R^{n-1}}\int_{\R} y_n e^{-h_{p,K}(\eta, y_n)} e^{\langle \xi_0, \eta\rangle} e^{t_0 y_n}\dif y_n\dif \eta \\
            &= \int_{\R^{n-1}} \int_{0}^\infty y_n e^{-h_{p,K}(\eta, y_n)} e^{\langle \xi_0,\eta\rangle} e^{t_0 y_n}\dif y_n\dif \eta
            \\ &\hspace{.5cm}+\int_{\R^{n-1}} \int_{-\infty}^0 y_n e^{-h_{p,K}(\eta, y_n)} e^{\langle \xi_0,\eta\rangle} e^{t_0 y_n}\dif y_n\dif \eta\\
            &= \int_{\R^{n-1}} \int_{0}^\infty y_n e^{-h_{p,K}(\eta, y_n)} e^{\langle \xi_0,\eta\rangle} e^{t_0 y_n}\dif y_n\dif \eta
            \\ &\hspace{.5cm}-\int_{\R^{n-1}} \int_{0}^\infty y_n e^{-h_{p,K}(\eta, -y_n)} e^{\langle \xi_0,\eta\rangle} e^{-t_0 y_n}\dif y_n\dif \eta 
            \\
            &= \int_{\R^{n-1}} \int_{0}^\infty y_n e^{-h_{p,K}(\eta, y_n)} e^{\langle \xi_0,\eta\rangle} (e^{t_0 y_n}-e^{- t_0y_n})\dif y_n\dif \eta\\
            &= \int_{\R^{n-1}} \int_{0}^\infty y_n e^{-h_{p,K-(\xi_0,0)}(\eta, y_n)}  2\sinh(t_0 y_n)\dif y_n\dif \eta, 
    \end{align*}}because $K$ is symmetric with respect to $e_n$, thus by Lemma \ref{ACLemma}, $h_{p,K}(\eta, -y_n)= h_{p,K}(\eta, y_n)$.
By Lemma \ref{ACLemma}, $2 y_n \sinh(t_0 y_n)$ vanishes almost everywhere in $\R^{n-1}\times (0,\infty)$, hence $t_0=0$.

In general, for $u\in\partial B_2^n$, let $A\in O(n)$ such that $Au= e_n$. If $K$ is symmetric with respect to $u$, then $A K$ is symmetric with respect to $Au= e_n$. By Lemma \ref{LpSantaloGL}, $A s_p(K)= s_p(AK)\in e_n^\perp= (Au)^\perp$. To finish the proof, since for $A\in O(n)$, $(A^T)^{-1}= A$, is remains to show
\begin{equation}\label{uperpAffine}
    (Au)^\perp= (A^T)^{-1} u^\perp,
\end{equation}
Since $A$ is invertible, it is enough to show $(Au)^\perp\subset (A^T)^{-1}u^\perp$.
For $x\in (Au)^\perp$, $0= \langle x, A u\rangle= \langle A^T x, u\rangle$, thus $A^T x\in u^\perp$, or $x\in (A^T)^{-1} u. $
\end{proof}

\begin{proof}[Proof of Lemma \ref{SantaloPointLemma}]
By Lemma \ref{SteinerProperties}, $\sigma_uK$ is symmetric with respect to $u^\perp$, thus by Lemma \ref{LpSymSP}, $s_p(\sigma_u K)\in u^\perp$.     
\end{proof}

\subsection{\texorpdfstring{Separating the $L^p$-polar}{Separating the Lᵖ-polar}}
\label{separatingSection}
We now turn to the proof of Lemma \ref{SeparatingLemma}.

After proving Lemma \ref{VolumeLemma} for $p=\infty$, in order to show the classical Santal\'o inequality, Meyer--Pajor translate $K$ so that $u^\perp$ $1/2$-separates it. Since $(K^{\circ})^\circ= K$, they switch to working with $K^\circ$, so that $u^\perp$ 1/2-separates its polar $(K^\circ)^\circ$ \cite[pp. 88--89]{meyer-pajor}. Then, by Lemma \ref{VolumeLemma} for $\lambda=1/2$, 
\begin{equation*}
    \inf_{x\in \R^n}\M(K-x)\leq \M(K)= \M(K^\circ)\leq \M(\sigma_{u}(K^\circ)). 
\end{equation*}
Doing this $n$ times for $n$ orthogonal directions produces a symmetric convex body $S$ with $\inf_{x\in\R^n}\M(K-x)\leq \M(S)$. It is, therefore, enough to prove Santal\'o's inequality for symmetric convex bodies. 

$L^p$-polarity is not a duality operation, so a similar argument would not work in our case. However, even in the Meyer--Pajor argument, it is not necessary to switch to $K^\circ$. Let us explain why. One may translate $K$ in the direction of $u$ such that $u^\perp$ $1/2$-separates $K^\circ$. Take $u=e_n$ for a moment. 
Starting with a convex body $K\subset \R^n$ such that $0\in\mathrm{int}\,K$, the volume of the polar $K^\circ$ is finite (Lemma \ref{finitenessLemma}). In particular,
\begin{equation*}
    |K^\circ \cap (e_n^\perp)^+|<\infty \quad \text{ and } \quad |K^\circ\cap (e_n^\perp)^-|<\infty. 
\end{equation*}
Moving $K$ in the direction of $e_n$ will always keep one of the two finite, 
\begin{equation}\label{polarityInversionEq}
    |(K-te_n)^\circ\cap (e_n^\perp)^-|<\infty \quad \text{ for all } t\in \R \text{ with } K-t e_n\subset (e_n^\perp)^-,
\end{equation}
even though, by Lemma \ref{finitenessLemma}, when the origin is no longer in the interior their sum will become infinite.
In other words, if $K$ mostly lies in $(e_n^\perp)^-$, then $K^\circ$ mostly lies in $(e_n^\perp)^+$.
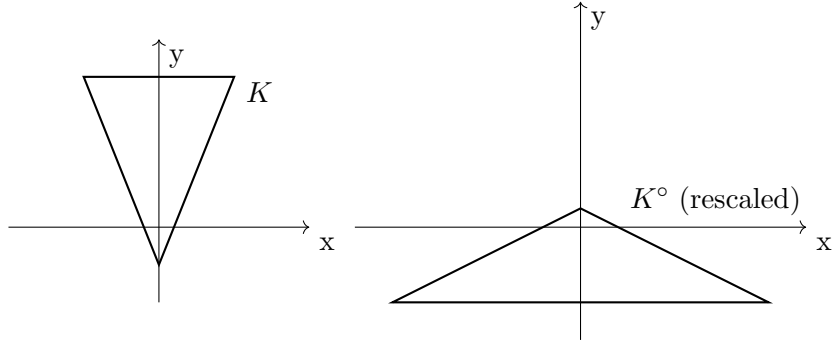
\begin{figure}[H]
    \centering
\begin{tikzpicture}[baseline, scale=1]
\draw[thick] (-1,2) -- (1,2) -- (0,-.5) -- cycle;
\draw[->] (-2,0) -- (2,0) node[anchor=north west] {x};
\draw[->] (0,-1) -- (0,2.5) node[anchor=north west] {y};
\node[above right=1pt] at (1, 1.5){$K$};
\end{tikzpicture}  
\begin{tikzpicture}[baseline, scale=.5]
\draw[thick] (0,1/2) -- (5,-2) -- (-5,-2) -- cycle;
\draw[->] (-6,0) -- (6,0) node[anchor=north west] {x};
\draw[->] (0,-3) -- (0,6) node[anchor=north west] {y};
\node[above right=1pt] at (1, 0){$K^\circ$ (rescaled)};
\end{tikzpicture}  
    \caption{If $K$ mostly lies in $(e_n^\perp)^+$, then $K^{\circ}$ mostly lies in $(e_n^\perp)^-$.}
\end{figure}

To prove \eqref{polarityInversionEq}, it is necessary to start with $K$ with $0\in\mathrm{int}\,K$, so that $|K^\circ|$ is finite (see Remark \ref{0intRemark} below). In particular, $K^\circ$ is bounded, and hence, there is $M>0$ such that
\begin{equation*}
    K^\circ\subset M B_2^n.
\end{equation*}
Then, for $t\in \R$ with $K-te_n\subset (e_n^\perp)^-$, 
\begin{equation}\label{coneSubset}
    (K-te_n)^\circ \subset \{(\eta, y_n)\in \R^{n-1}\times \R: y_n>  -1/t, |\eta|\leq M \sqrt{1+ y_n t}\},
\end{equation}
i.e., $(K-te_n)^\circ$ is contained in a cone whose intersection with $(e_n^\perp)^-$ is of finite volume. Indeed, for $(\eta, y_n)\in (K-te_n)^\circ$, 
\begin{equation}\label{polarConeEq}
    \langle (\xi,x_n-t), (\eta,y_n)\rangle= \langle (\xi,x_n), (\eta,y_n)\rangle- ty_n \leq 1, \quad \text{ for all } (\xi,x_n)\in K,
\end{equation}
or, equivalently, $\langle (\xi,x_n), (\eta, y_n)\rangle\leq 1+t y_n$. This forces $y_n> -1/t$, because if $y_n\leq -1/t$, or equivalently, $1+ty_n\leq 0$, choosing $\e>0$ such that $(0,-\e)\in K$ (such an $\e$ exists because $0\in\mathrm{int}\,K$), \eqref{polarConeEq} gives $-\e y_n\leq 1+ ty_n\leq 0$, which in turn implies $y_n\geq 0$, a contradiction to $y_n\leq -1/t$. Thus, $1+t y_n>0$. As a result, by \eqref{polarConeEq},
\begin{equation*}
    \frac{(\eta,y_n)}{1+t y_n} \in K^\circ,
\end{equation*}
which is bounded by $M$, thus
\begin{equation*}
    \frac{|\eta|}{\sqrt{1+t y_n}}\leq \frac{\sqrt{|\eta|^2+ |y_n|^2}}{\sqrt{1+t y_n}}\leq M.
\end{equation*}
proving \eqref{coneSubset}. The part of the cone of the right-hand side of \eqref{coneSubset} that lies in $(e_n^\perp)^-$ has finite volume. Indeed, it is contained in the cylinder $(M B_2^{n-1})\times [-1/t,0]$.

In sum, we have verified the following claim for $p=\infty$. 
\begin{claim}\label{FinitenessClaim}
Let $p\in (0,\infty]$. For a convex body $K\subset \R^n$ with $0\in\mathrm{int}\, K$, and $t\in \R$ such that $(K- te_n)\subset (e_n^\perp)^-$,
\begin{equation*}
    |(K-te_n)^{\circ,p} \cap (e_n^\perp)^-|< \infty. 
\end{equation*}
\end{claim}

To prove Claim \ref{FinitenessClaim} for finite $p$, we need an approximation of $K^{\circ,p}$ by $K^\circ$. 

\begin{lemma}{\upshape\cite[Lemma 2.6]{BMR}}
\label{hKhpKApprox}
    Let $p\in (0,\infty)$. For a convex body $K\subset \R^n$ with $b(K)=0$, and $\lambda\in (0,1)$, 
    \begin{equation*}
        h_K(y)\leq h_{p,K}(y/\lambda)-\frac{n}{p}\log(1-\lambda), 
    \end{equation*}
    for all $y\in \R^n$.
\end{lemma}

\begin{corollary}\label{InclusionLemma}
    Let $p\in(0,\infty]$. For a convex body $K\subset \R^n$, 
    \begin{equation*}
        K^{\circ,p}\subset \frac{1}{\lambda (1-\lambda)^p} \left(K-\Big( 1-\frac{1}{\lambda}\Big) b(K) \right)^\circ,
    \end{equation*}
    for all $\lambda\in (0,1)$.
\end{corollary}
\begin{proof}[Proof of Corollary \ref{InclusionLemma}]
    Let $\lambda\in (0,1)$. By Lemma \ref{hKhpKApprox}, 
    \begin{equation}\label{hKhpKineq}
        h_{K-b(K)}(y)\leq h_{p, K-b(K)}(y/\lambda) -\frac{n}{p}\log(1-\lambda), \quad \text{ for all } y\in\R^n. 
    \end{equation}
    Since $h_{p,K-b(K)}(y)= h_{p,K}(y)- \langle y,b(K)\rangle$ and $h_{K-b(K)}(y)= h_K(y)-\langle y, b(K)\rangle$ \cite[Lemma 2.2 (ii)]{BMR}, \eqref{hKhpKineq} gives
    \begin{equation*}
       h_{K-(1-\frac{1}{\lambda})b(K)}(y)= h_{K}(y)-\left( 1-\frac{1}{\lambda}\right)\langle y,b(K)\rangle \leq h_{p,K}(y/\lambda) -\frac{n}{p}\log(1-\lambda). 
    \end{equation*}
    So, for the norm \eqref{KcircpNorm}, 
    \begin{equation*}
        \|y\|_{(K-(1-\frac1\lambda) b(K))^{\circ}}\leq \frac{1}{\lambda (1-\lambda)^p}\|y\|_{K^{\circ,p}},
    \end{equation*}
    from which the claim follows. 
\end{proof}

\begin{proof}[Proof of Claim \ref{FinitenessClaim}]
The $p=\infty$ case follows from \eqref{coneSubset}. Let $p\in (0,\infty)$, and 
\begin{equation*}
\begin{gathered}
    a\defeq \max\{t\in \R: K-te_n\subset (e_n^\perp)^+\},\\
    b \defeq \min\{t\in \R: K-te_n\subset (e_n^\perp)^-\}.
\end{gathered}
\end{equation*}
Since $0\in\mathrm{int}\,K$, $a<0$ and $b>0$.
\begin{figure}[H]
    \centering
   \begin{tikzpicture}[scale=1]
\draw[thick] (-1+1/7,0-3/7) -- (0+1/7,1-3/7) -- (2+1/7,-3/7) -- cycle;
\draw[dashed] (-1.5, 1-3/7) -- (2.7, 1-3/7);
\draw[dashed] (-1.5, -3/7) -- (2.7, -3/7);
\draw[->] (-2.5,0) -- (3,0) node[anchor=north west] {x};
\draw[->] (0,-1) -- (0,1.5) node[anchor=north west] {y};
\node[below right=1pt] at (0, -.5){$a$};
\node[above right=1pt] at (0, .6){$b$};
\node[above right=1pt] at (-1, 0){$K$};
\fill (0, -3/7)  circle[radius= 1pt];
\fill (0, 1-3/7)  circle[radius= 1pt];
\end{tikzpicture}  
    \caption{$a$ and $b$.}
    \label{fig:ab}
\end{figure}
For $t\in (a,b)$, $te_n \in\mathrm{int}\,K$, thus $ 0\in\mathrm{int}\,(K-te_n)$.
By Lemma \ref{finitenessLemma}, $\M_p(K-te_n)<\infty$ for all $t\in (a,b)$. In particular, $|(K-te_n)^{\circ,p}\cap (e_n^\perp)^-|<\infty$ for all $t\in (a,b)$. By continuity, it is then enough to prove the claim for $t>b$. In that case, since $t>b$, for $\lambda_0\in(0,1)$ close enough to $1$, if
\begin{equation*}
    L\defeq (K-te_n)- \left( 1-\frac{1}{\lambda_0}\right) b(K-te_n), 
\end{equation*}
then $L\subset (e_n^\perp)^-$. 
By the $p=\infty$ case, $|L^\circ\cap (e_n^\perp)^-|<\infty$. 
By Corollary \ref{InclusionLemma},
 \begin{equation*}
     (K-t e_n)^{\circ,p}\cap (e_n^\perp)^-\subset \frac{L^\circ\cap (e_n^\perp)^-}{\lambda_0 (1-\lambda_0)^{p}}, 
 \end{equation*}
 from which the claim follows.
\end{proof}

Replacing $e_n$ by $-e_n$ proves the following. 
\begin{claim}\label{FinitenessClaim2}
    Let $p\in (0,\infty]$. For a convex body $K\subset \R^n$ with $0\in\mathrm{int}\,K$, and $t\in \R$ such that $(K-te_n)\subset (e_n^\perp)^+$, 
    \begin{equation*}
        |(K-te_n)^{\circ,p}\cap (e_n^\perp)^+|<\infty. 
    \end{equation*}
\end{claim}

\begin{remark}\label{0intRemark}
For Claims \ref{FinitenessClaim} and \ref{FinitenessClaim2}, it is important to start with a convex body $K$ with $0\in \mathrm{int}\,K$. Otherwise, the claim does not hold. For example, the polar body of the standard simplex in $\R^2$, 
\begin{equation*}
    \Delta_2\defeq \{(x,y)\in[0,\infty)^2: x+y\leq 1\}, 
\end{equation*}
translated in the direction of the $y$-axis, for $t\in (0,1)$,
\begin{equation*}
    (\Delta_2- te_2)^\circ= \left\{(x,y)\in\R^2: -\frac1t\leq y\leq \frac{1}{1-t} \text{ and } x-ty\leq 1\right\}.
\end{equation*}
Both $(\Delta_2-te_2)^\circ \cap (e_2^\perp)^+$ and $(\Delta_2- te_2)^\circ\cap (e_2^\perp)^-$ has infinite volume.
\begin{figure}[H]
    \centering
    \begin{tikzpicture}[scale=1.2, baseline]
\draw[thick] (0,-1/2) -- (1,-1/2) -- (0,1/2) -- cycle;
\draw[fill=gray!50!white] plot[smooth,samples=100,domain=1:0] (\x,-1/2) -- 
plot[smooth,samples=100,domain=0:1] (\x,1/2-\x);
\draw[->] (-1,0) -- (2,0) node[anchor=north west] {x};
\draw[->] (0,-1) -- (0,3/2) node[anchor=north west] {y};
\node[above right=1pt] at (.3, .05){$\Delta_2-1/2 e_2$};
\end{tikzpicture}  
    \begin{tikzpicture}[scale=1.2, baseline]
\draw[thick] (-2,1/2) -- (5/4,1/2) -- (3/4,-1/2) -- (-2, -1/2);
\fill[gray!50]      (-2,1/2) -| (5/4,1/2) -| (3/4,-1/2) -| (-2, -1/2) -- cycle;
\draw[fill=gray!50!white] plot[smooth,samples=100,domain=3/4:5/4] (\x,2*\x-2) -- 
plot[smooth,samples=100,domain=5/4:3/4] (\x,1/2);
\draw[->] (-2,0) -- (2,0) node[anchor=north west] {x};
\draw[->] (0,-1) -- (0,1.5) node[anchor=north west] {y};
\node[above right=1pt] at (-.9, .55){$(\Delta_2-1/2 e_2)^\circ$};
\end{tikzpicture}  
    \caption{$\Delta_2-1/2 e_2$ and its polar $(\Delta_2-1/2 e_2)^\circ$.}
\end{figure}
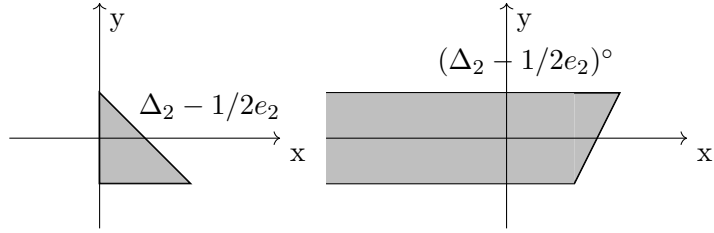
\end{remark}

\begin{proof}[Proof of Lemma \ref{SeparatingLemma}]
Let
\begin{equation*}
    \begin{gathered}
    a\defeq \max\{t<0: K-te_n\subset (e_n^\perp)^+\}, \\
    b\defeq \min\{t>0: K-te_n\subset (e_n^\perp)^-\}
    \end{gathered}
\end{equation*}
(Figure \ref{fig:ab}). 
Consider
\begin{equation*}
    f(t)\defeq \frac{|(K-te_n)^{\circ,p}\cap (e_n^\perp)^+|}{|(K-te_n)^{\circ,p}\cap (e_n^\perp)^-|}, \quad t\in [a,b]. 
\end{equation*}
By \cite[Lemma 4.3]{BMR}, $f$ is finite valued and $C^2$ for $t\in (a,b)$. By Lemma \ref{finitenessLemma}, $|(K-ae_n)^{\circ,p}|=\infty$ and $|(K-be_n)^{\circ,p}|=\infty$.
In addition, by Claim \ref{FinitenessClaim2}, $|(K-a e_n)^{\circ,p}\cap (e_n^\perp)^+|<\infty$, forcing $|(K-ae_n)^{\circ,p}\cap (e_n^\perp)^-|=\infty$ thus $f(a)=0$. Similarly, by Claim \ref{FinitenessClaim}, $|(K-b e_n)^{\circ,p}\cap (e_n^\perp)^-|<\infty$ and $|(K-be_n)^{\circ,p}\cap (e_n^\perp)^+|=\infty$, thus $f(b)= \infty$. By the intermediate valued theorem, there is $t_0\in (a,b)$ such that $f(t_0)= \frac{\lambda}{1-\lambda}$ (could also choose $\frac{1-\lambda}{\lambda}$), i.e., $e_n^\perp$ $\lambda$-separates $(K-t_0e_n)^{\circ,p}$. 
\end{proof}

\subsection{A generalization Ball's Brunn--Minkowski inequality for the harmonic mean}
\label{BallSection}
For $p=\infty$, $\|\cdot\|_{K^\circ}= h_K$, thus going from inequalities on $h_K$ to inclusions on $K^\circ$ is immediate. A similar equality does not hold for finite $p$, so a little more work is required.
The following theorem due to Ball is very useful for this purpose \cite[Theorem 4.10]{ball}. 
\begin{theorem}\label{BallIneq}
    Let $F,G,H: (0,\infty)\to [0,\infty)$ be measurable functions, not almost everywhere $0$, so that
    \begin{equation*}
        H(r)\geq F(t)^{\frac{s}{t+s}} G(s)^{\frac{t}{t+s}}, \quad \text{ for all } \quad \frac{2}{r}=\frac{1}{t}+\frac{1}{s}.
    \end{equation*}
    Then, for $q\geq 1$, 
    \begin{equation*}
        \left(\int_{0}^\infty r^{q-1} H(r) \dif r \right)^{-\frac1q}\leq \frac12\left(\int_0^\infty t^{q-1} F(t) \dif t \right)^{-\frac1q}+\frac12 \left( \int_0^\infty s^{q-1} G(s)\dif s\right)^{-\frac1q}.
    \end{equation*}
\end{theorem}

As a corollary we obtain the following. 
\begin{corollary}\label{GeneralizedBall}
    Let $\lambda\in [0,1]$, and $F,G,H: (0,\infty)\to [0,\infty)$ be measurable functions, not almost everywhere $0$, so that
    \begin{equation*}
        H(r)\geq F(t)^{\frac{(1-\lambda)s}{\lambda t+ (1-\lambda)s}} G(s)^{\frac{\lambda t}{\lambda t+ (1-\lambda)s}}, \quad \text{ for all } \quad \frac{1}{r}= \frac{1-\lambda}{t}+ \frac{\lambda}{s}.
    \end{equation*}
    Then, 
    \begin{equation*}
        \left(\int_{0}^\infty r^{q-1} H(r) \dif r \right)^{-\frac1q}\leq (1-\lambda)\left(\int_0^\infty t^{q-1} F(t) \dif t \right)^{-\frac1q}+\lambda \left( \int_0^\infty s^{q-1} G(s)\dif s\right)^{-\frac1q}.
    \end{equation*}
\end{corollary}
\begin{proof}
When $\lambda=0$ or $1$, the statement follows from the monotonicity of the integral. For $\lambda\in (0,1)$, 
let
\begin{equation*}
    \tau= \frac{t}{2(1-\lambda)}, \quad \text{ and } \quad \sigma= \frac{s}{2\lambda}. 
\end{equation*}
Then, 
\begin{equation*}
    \frac{(1-\lambda)s}{\lambda t+ (1-\lambda) s}= \frac{(1-\lambda) 2\lambda \sigma}{\lambda 2 (1-\lambda) \tau+ (1-\lambda) 2\lambda \sigma}= \frac{\sigma}{\tau+ \sigma}, 
\end{equation*}
and
\begin{equation*}
    \frac{\lambda t}{\lambda t+ (1-\lambda)s}= \frac{\lambda 2 (1-\lambda)\tau}{\lambda 2 (1-\lambda)\tau+ (1-\lambda) 2\lambda \sigma}= \frac{\tau}{\tau+ \sigma}.
\end{equation*}
Therefore, if
\begin{equation*}
    f(\tau)\defeq F(2(1-\lambda)\tau), \quad \text{ and } \quad g(\sigma)\defeq G(2\lambda \sigma), 
\end{equation*}
then
\begin{equation*}
    H(r)\geq f(\tau)^{\frac{\sigma}{\tau+\sigma}} g(\sigma)^{\frac{\tau}{\tau+\sigma}}, \text{ for all } \quad \frac{2}{r}= \frac{1}{\tau}+\frac{1}{\sigma}.  
\end{equation*}
By Theorem \ref{BallIneq}, for all $q\geq 1$, 
\begin{equation*}
\begin{aligned}
    &\left(\int_0^\infty r^{q-1} H(r)\dif r \right)^{-\frac1q}\\
    &\leq \frac12    \left(\int_0^\infty \tau^{q-1} f(\tau)\dif \tau \right)^{-\frac1q}+ \frac12    \left(\int_0^\infty \sigma^{q-1} g(\sigma)\dif \sigma \right)^{-\frac1q}\\
    &= \frac12   \left(\int_0^\infty \tau^{q-1} F(2(1-\lambda)\tau)\dif \tau \right)^{-\frac1q}+ \frac12    \left(\int_0^\infty \sigma^{q-1} G(2\lambda \sigma)\dif \sigma \right)^{-\frac1q}\\
    &= \frac12   \left(\int_0^\infty \left( \frac{t}{2(1-\lambda)}\right)^{q-1} F(t)\frac{\dif t}{2(1-\lambda)} \right)^{-\frac1q} + \frac12    \left(\int_0^\infty \left(\frac{s}{2\lambda}\right)^{q-1} G(s)\frac{\dif s}{2\lambda} \right)^{-\frac1q}\\
    &=  (1-\lambda) \left(\int_0^\infty t^{q-1} F(t)\dif t \right)^{-\frac1q}+ \lambda \left(\int_0^\infty s^{q-1} G(s)\dif s \right)^{-\frac1q},
\end{aligned}
\end{equation*}
concluding the proof.  
\end{proof}

\subsection{\texorpdfstring{Monotonicity of the volume of the slices of the $L^p$-polar under Steiner symmetrization}{Monotonicity of the volume of the slices of the Lᵖ-polar under Steiner symmetrization}}
\label{SteinerMonotonicitySection}
The method of proof for Lemma \ref{VolumeLemma} is standard: compare the volume of the slices of $K^\circ$ and $(\sigma_uK)^{\circ,p}$ normal to the direction of the symmetrization. For a convex body $K\subset \R^n$ let
\begin{equation}\label{sliceDef}
    K(x_n)\defeq \{\xi\in \R^{n-1}: (\xi, x_n)\in K\}, 
\end{equation}
the `slice' of $K$ at height $x_n$. By Tonelli's theorem \cite[\S 2.37]{folland}, 
\begin{equation}\label{sliceVolumeIntegral}
    |K|= \int_K \dif x= \int_{\{(\xi,x_n)\in \R^{n-1}\times \R: \xi\in K(x_n)\}}\dif \xi\dif x_n= \int_{-\infty}^{\infty} |K(x_n)|\dif x_n. 
\end{equation}
Therefore, to obtain Lemma \ref{VolumeLemma}, we will use the following. 
\begin{lemma}
\label{volumeSliceComparison}
Let $p\in [0,\infty)$. Let, also, $\lambda\in (0,1)$ and $K\subset \R^n$ a convex body for which $e_n^\perp$ $\lambda$-separates $K^{\circ,p}$. Then, 
\begin{equation*}
    |(\sigma_{e_n}K)^{\circ,p}(r)|\geq |K^{\circ,p}(t)|^{\frac{s}{t+s}} |K^{\circ,p}(-s)|^{\frac{t}{t+s}}, 
\end{equation*}
for all $r,t,s>0$ with $\frac{2}{r}=\frac{1}{t}+\frac{1}{s}$. 
\end{lemma}

To prove Lemma \ref{volumeSliceComparison} we apply the multiplicative Brunn--Minkowski inequality on
\begin{equation}
\label{SetInclusion}
    \frac{s}{t+s} K^{\circ,p}(t)+ \frac{t}{t+s} K^{\circ,p}(-s)\subset (\sigma_{e_n}K)^{\circ,p}(r), \quad \text{ for all } \quad \frac{2}{r}=\frac{1}{t}+\frac1s,
\end{equation}
Before proving \eqref{SetInclusion}, note that the
symmetry of $((\sigma_{e_n}K)^{\circ,p}$ with respect to $e_n^\perp$ (Lemma \ref{LpPolarSteinerSym}), allows for 
the right-hand side of \eqref{SetInclusion} to be written as 
\begin{equation*}
    (\sigma_{e_nK})^{\circ,p}(r)=\frac{s}{t+s} (\sigma_{e_n}K)^{\circ,p}(t)+ \frac{t}{t+s} (\sigma_{e_n}K)^{\circ,p}(-s),
\end{equation*}
because $(\sigma_{e_n}K)^{\circ,p}(-s)= (\sigma_{e_n}K)^{\circ,p}(s)$, and $r=\frac{2ts}{t+s}$. Thus, \eqref{SetInclusion} is a monotonicity property for the slices of $K^{\circ,p}$ under Steiner symmetrization. 

The proof of \eqref{SetInclusion} follows from an inequality of norms
\begin{equation}\label{normInclusion}
    \left\|\left( \frac{s}{t+s}\xi+\frac{t}{t+s}\xi', r\right)\right\|_{(\sigma_{e_n}K)^{\circ,p}}\leq \frac{s}{t+s}\|(\xi,t)\|_{K^{\circ,p}}+ \frac{t}{t+s}\|(\xi', -s)\|_{K^{\circ,p}},
\end{equation}
for all $r,t,s>0$ with $\frac2r= \frac1t+\frac1s$.
Now, to prove \eqref{normInclusion}, we use Corollary \ref{GeneralizedBall} on the $L^p$-support functions, hence we require the next lemma proved in \S \ref{SupportIneqSection} below.  

\begin{lemma}
\label{technicalLemma}
    Let $p\in (0,\infty]$, $\xi,\xi'\in \R^{n-1}$, $t,s,r>0$ with $\frac2r=\frac1t+\frac1s$, and $\tau\defeq \frac{t}{t+s}$. For all $\alpha,\beta,\gamma>0$ with $\frac{1}{\gamma}=\frac{1-\tau}{\alpha}+\frac{\tau}{\beta}$, 
    \begin{equation*}
        \begin{aligned}
            &h_{p,\sigma_{e_n}K}\Big( \gamma((1-\tau) \xi +\tau \xi', r)\Big)\\
            &\leq \frac{(1-\tau)\beta}{\tau \alpha+ (1-\tau)\beta} h_{p, K}(\alpha(\xi, t))
            + \frac{\tau \alpha}{\tau\alpha+ (1-\tau)\beta} h_{p,K}(\beta(\xi', -s)).
        \end{aligned}
    \end{equation*}
\end{lemma}

Indeed, applying Corollary \ref{GeneralizedBall} with $q=n$ on Lemma \ref{technicalLemma} gives
\begin{equation*}
    \begin{aligned}
       \Big(&(n-1)!\Big)^{-\frac1n}\|((1-\tau)\xi+\tau\xi', r)\|_{(\sigma_{e_n}K)^{\circ,p}}\\
       &=\left(\int_{0}^\infty \gamma^{n-1} e^{-h_{p,\sigma_{e_n}K} (\gamma ((1-\tau)\xi+\tau\xi', r))}\dif \gamma \right)^{-\frac1n} \\
       &\leq (1-\tau) \left(\int_{0}^\infty \alpha^{n-1} e^{-h_{p,K}(\alpha(\xi,t))}\dif \alpha \right)^{-\frac1n}+ \tau \left( \int_{0}^\infty\beta^{n-1} e^{-h_{p,K}(\beta(\xi',s))}\dif \beta\right)^{-\frac1n} \\
       &= \Big((n-1)!\Big)^{-\frac1n} (1-\tau)\|(\xi,t)\|_{K^{\circ,p}}+ \Big((n-1)!\Big)^{-\frac1n} \tau \|(\xi',s)\|_{K^{\circ,p}}, 
    \end{aligned}
\end{equation*}
proving \eqref{normInclusion}, because $\tau= \frac{t}{t+s}$.

For \eqref{SetInclusion}, given $\xi\in K^{\circ,p}(t)$ and $\xi'\in K^{\circ,p}(-s)$, by definition \eqref{sliceDef}, $(\xi, t)\in K^{\circ,p}$ and $(\xi',-s)\in K^{\circ,p}$, i.e., $\|(\xi,t)\|_{K^{\circ,p}}\leq 1$ and $\|(\xi', -s)\|_{K^{\circ,p}}\leq 1$. By \eqref{normInclusion}, since $r=\frac{2ts}{t+s}$,
\begin{equation*}
\begin{aligned}
    \left\| \left( \frac{s}{t+s}\xi+ \frac{t}{t+s}\xi', r\right)\right\|_{(\sigma_{e_n}K)^{\circ,p}}&\leq \frac{s}{t+s} \|(\xi, t)\|_{K^{\circ,p}}+ \frac{t}{t+s}\|(\xi',-s)\|_{K^{\circ,p}}\\
    &\leq \frac{s}{t+s}+ \frac{t}{t+s}=1
\end{aligned}
\end{equation*}
thus $(\frac{s}{t+s}\xi+ \frac{t}{t+s}\xi', r)\in (\sigma_{e_n}K)^{\circ,p}$ and $\frac{s}{t+s}\xi+ \frac{t}{t+s}\xi'\in (\sigma_{e_n}K)^{\circ,p}(r)$.

\begin{proof}[Proof of Lemma \ref{volumeSliceComparison}]
    By the multiplicative Brunn--Minkowski inequality applied on \eqref{SetInclusion}, 
    \begin{equation*}
    \begin{aligned}
        |(\sigma_{e_n}K)^{\circ,p}(r)|\geq \left| \frac{s}{t+s} K^{\circ,p}(t)+ \frac{t}{t+s} K^{\circ,p}(-s)\right|\geq |K^{\circ,p}(t)|^{\frac{s}{t+s}} |K^{\circ,p}(-s)|^{\frac{t}{t+s}}.
    \end{aligned}
    \end{equation*}
\end{proof}

\begin{proof}[Proof of Lemma \ref{VolumeLemma}]
To begin with, take $u=e_n$.
Let
\begin{equation*}
    \begin{gathered}
        F(t)\defeq |K^{\circ,p}(t)|, \quad t\in (0,\infty), \\
        G(s)\defeq |K^{\circ,p}(-s)|, \quad s\in (0,\infty), \\
        H(r)\defeq |(\sigma_{e_n}K)^{\circ,p}(r)|, \quad r\in (0,\infty).
    \end{gathered}
\end{equation*}
By Lemma \ref{volumeSliceComparison},  $H(r)\geq F(t)^{\frac{s}{t+s}} G(s)^{\frac{t}{t+s}}$, for all $t,s,r>0$ with $\frac2r=\frac1t+\frac1s$. Therefore, by Theorem \ref{BallIneq} for $q=1$, and \eqref{sliceVolumeIntegral},
{\allowdisplaybreaks
\begin{align*}
    \frac{4}{|(\sigma_{e_n}K)^{\circ,p}|}=
    \frac{2}{\int_0^\infty |(\sigma_{e_n}K)^{\circ,p}(r)|\dif r}&\leq \frac{1}{\int_0^\infty |K^{\circ,p}(t)|\dif t}+ \frac{1}{\int_0^\infty |K^{\circ,p}(-s)|\dif s} \\
    &= \frac{1}{\int_0^\infty |K^{\circ,p}(t)|\dif t}+ \frac{1}{\int_{-\infty}^0 |K^{\circ,p}(s)|\dif s} \\
    &= \frac{1}{|K^{\circ,p}\cap (e_n^\perp)^+|}+ \frac{1}{|K^{\circ,p}\cap (e_n^\perp)^-|}\\
    &= \frac{|K^{\circ,p}\cap (e_n^\perp)^+|+ |K^{\circ,p}\cap (e_n^\perp)^-|}{|K^{\circ,p}\cap (e_n^\perp)^+||K^{\circ,p}\cap (e_n^\perp)^-|} \\
    &= \frac{|K^{\circ,p}|}{|K^{\circ,p}\cap (e_n^\perp)^+||K^{\circ,p}\cap (e_n^\perp)^-|} \\
    &= \frac{1}{\lambda(1-\lambda)|K^{\circ,p}|}, 
\end{align*}
}
the last equality is because, by assumption, $e_n^\perp$ $\lambda$-separates $K^{\circ,p}$.

In general, let $A\in O(n)$ such that $u= Ae_n$. If $u$ $\lambda$-separates $K^{\circ,p}$, then $Au= e_n$ $\lambda$-separates $(AK)^{\circ,p}$. Indeed, if $|K^{\circ,p}\cap (u^\perp)^+|= \lambda |K^{\circ,p}|$, then, by \eqref{KcircpGLnR} and \eqref{uperpAffine},
\begin{equation*}
    (AK)^{\circ,p}\cap ((Au)^\perp)^+= \Big( (A^{-1})^T K^{\circ,p}\Big)\cap \Big((A^{-1})^T (u^\perp)^+\Big)= (A^{-1})^T (K^{\circ,p}\cap (u^\perp)^+).
\end{equation*}
Thus, $|(AK)^{\circ,p}\cap ((Au)^\perp)^+|= |\det A| |K^{\circ,p} \cap (u^\perp)^+|= |\det A| \lambda |K^{\circ,p}|= \lambda |(AK)^{\circ,p}|$, i.e., $Au$ $\lambda$-separates $(AK)^{\circ,p}$. By Lemma \ref{steiner_orthogonal}, since $Au= e_n$, and $e_n$ $\lambda$-separates $(AK)^{\circ,p}$,
\begin{equation*}
\begin{aligned}
    |(\sigma_{u}K)^{\circ,p}|&= |(A^{-1}\sigma_{e_n}(AK))^{\circ,p}|\\
    &= |A^T (\sigma_{e_n}(AK))^{\circ,p}|\\
    &= |\det A| |(\sigma_{e_n}(AK))^{\circ,p}| \\
    &\geq 4\lambda(1-\lambda) |\det A| |(AK)^{\circ,p}|\\
    &= 4\lambda (1-\lambda) |K^{\circ,p}|,
\end{aligned}
\end{equation*}
because by \eqref{KcircpGLnR}, $|\det A||(AK)^{\circ,p}|= |\det A||(A^{-1})^T K^{\circ,p}|= |K^{\circ,p}|$. 
\end{proof}

\subsection{\texorpdfstring{An inequality for the $L^p$-support functions}{An inequality for the Lp-support function}}
\label{SupportIneqSection}

\begin{proof}[Proof of Lemma \ref{technicalLemma}]
    Let $f,g: \pi_{e_n^\perp}(K)\to \R, g\leq f,$ such that
    \begin{equation*}
        K= \{(\xi,x_n)\in \R^{n-1}\times \R: (\xi,0)\in \pi_{e_n^\perp}(K) \text{ and } g(\xi)\leq x_n\leq f(\xi)\}. 
    \end{equation*}
    Then,
    \begin{equation*}
        \sigma_{e_n^\perp}K= \left\{(\xi,x_n)\in \R^{n-1}\times \R: (\xi,0)\in \pi_{e_n^\perp}(K)\text{ and } |x_n|\leq \frac{f(\xi)-g(\xi)}{2} \right\}.
    \end{equation*}
    Let
\begin{equation*}
    J(\eta, y_n)\defeq \frac{2\sinh(p y_n\frac{f(\eta)-g(\eta)}{2})}{p y_n}, \quad \eta\in \R^{n-1}, y_n\in \R. 
\end{equation*}
    Compute,
    \begin{equation}\label{Eq1}
        \begin{aligned}
            &h_{p, \sigma_{e_n}K}(\gamma((1-\tau)\xi+\tau \xi', r))\\
            &=\frac1p \log\int_{\pi_{e_n^\perp}(K)}\int_{-\frac{f(\eta)-g(\eta)}{2}}^{\frac{f(\eta)-g(\eta)}{2}} e^{p\gamma\langle (1-\tau)\xi+ \tau \xi', \eta\rangle} e^{p\gamma y_n r}\frac{\dif y_n\dif \eta}{|\sigma_{e_n}K|}\\
            &= \frac1p \log\int_{\pi_{e_n^\perp}(K)} e^{p\gamma (1-\tau)\langle \xi, \eta\rangle+ p\gamma\tau \langle \xi',\eta\rangle} \frac{1}{p\gamma r} (e^{p\gamma r \frac{f(\eta)-g(\eta)}{2}}- e^{-p\gamma r\frac{f(\eta)-g(\eta)}{2}})\frac{\dif \eta}{|K|} \\
            &= \frac1p \log\int_{\pi_{e_n^\perp}(K)} e^{p\gamma (1-\tau)\langle \xi, \eta\rangle+ p\gamma\tau \langle \xi',\eta\rangle} \frac{2\sinh(p\gamma r\frac{f(\eta)-g(\eta)}{2})}{p\gamma r} \frac{\dif \eta}{|K|}\\
            &= \frac1p \log\int_{\pi_{e_n^\perp}(K)} e^{p\gamma (1-\tau)\langle \xi, \eta\rangle+ p\gamma\tau \langle \xi',\eta\rangle} J(\eta, \gamma r) \frac{\dif \eta}{|K|}.
        \end{aligned}
    \end{equation}
Also, 
\begin{equation}\label{Eq2}
    \begin{aligned}
        h_{p,K}(\alpha(\xi,t))&= \frac1p\log\int_{\pi_{e_n^\perp}(K)}\int_{g(\eta)}^{f(\eta)} e^{p\alpha\langle \xi,\eta\rangle} e^{p\alpha ty_n}\frac{\dif y_n\dif\eta}{|K|} \\
        &= \frac1p \log\int_{\pi_{e_n^\perp}(K)} e^{p\alpha\langle \xi,\eta\rangle} \frac{1}{p\alpha t} (e^{p\alpha t f(\eta)}- e^{p\alpha t g(\eta)})\frac{\dif \eta}{|K|} \\
        &= \frac1p \log\int_{\pi_{e_n^\perp}(K)} e^{p\alpha\langle \xi,\eta\rangle} e^{p\alpha t\frac{f(\eta)+g(\eta)}{2}} \frac{2\sinh(p\alpha t\frac{f(\eta)-g(\eta)}{2})}{p\alpha t} \frac{\dif \eta}{|K|}  \\
        &= \frac1p \log\int_{\pi_{e_n^\perp}(K)} e^{p\alpha\langle \xi,\eta\rangle} e^{p\alpha t\frac{f(\eta)+g(\eta)}{2}} J(\eta, \alpha t)\frac{\dif \eta}{|K|}.
    \end{aligned}
\end{equation}
Similarly, 
\begin{equation}\label{Eq3}
    h_{p,K}(\beta(\xi',-s))=\frac1p \log\int_{\pi_{e_n^\perp}(K)} e^{p\beta\langle \xi',\eta\rangle} e^{-p\beta s\frac{f(\eta)+g(\eta)}{2}} J(\eta, \beta s) \frac{\dif \eta}{|K|}, 
\end{equation}
because $J$ is even in the second variable. By \eqref{Eq2}, \eqref{Eq3}, and H\"older's inequality, 
\begin{equation}\label{weirdBigIneq}
\begin{aligned}
        &\frac{(1-\tau)\beta}{\tau\alpha+ (1-\tau)\beta} h_{p,K}(\alpha(\xi,t))+ \frac{\tau\alpha}{\tau\alpha+(1-\tau)\beta}h_{p,K}(\beta(\xi',-s)) \\
        &= \frac1p \log\left[ \left( \int_{\pi_{e_n^\perp}(K)} e^{p\alpha\langle \xi,\eta\rangle} e^{p\alpha t\frac{f(\eta)+g(\eta)}{2}} J(\eta, \alpha t)\frac{\dif \eta}{|K|}\right)^\frac{(1-\tau)\beta}{\tau\alpha+ (1-\tau)\beta} \right.
        \\
        &\left.
        \hspace{1.6cm}
        \left( \int_{\pi_{e_n^\perp}(K)} e^{p\beta\langle \xi',\eta\rangle} e^{-p\beta s\frac{f(\eta)+g(\eta)}{2}} J(\eta, \beta s) \frac{\dif \eta}{|K|}\right)^{\frac{\tau\alpha}{\tau\alpha+(1-\tau)\beta}}
        \right] \\
        &\geq \frac1p \log \int_{\pi_{e_n^\perp}(K)} e^{p\frac{(1-\tau)\alpha \beta}{\tau\alpha+ (1-\tau)\beta}\langle \xi,\eta\rangle} e^{p\frac{(1-\tau)\alpha\beta}{\tau\alpha+ (1-\tau)\beta} t\frac{f(\eta)+g(\eta)}{2}} J(\eta, \alpha t)^{\frac{(1-\tau)\beta}{\tau\alpha+ (1-\tau)\beta}} \\
        &\hspace{2.75cm} e^{p{\frac{\tau\alpha\beta}{\tau\alpha+(1-\tau)\beta}}\langle \xi',\eta\rangle} e^{-p{\frac{\tau\alpha\beta}{\tau\alpha+(1-\tau)\beta}} s\frac{f(\eta)+g(\eta)}{2}} J(\eta, \beta s)^{{\frac{\tau\alpha}{\tau\alpha+(1-\tau)\beta}}}\frac{\dif\eta}{|K|} \\
        &= \frac1p \log \int_{\pi_{e_n^\perp}(K)} e^{p\gamma (1-\tau)\langle \xi,\eta\rangle} e^{p\gamma(1-\tau) t\frac{f(\eta)+g(\eta)}{2}} J(\eta, \alpha t)^{\frac{(1-\tau)\beta}{\tau\alpha+ (1-\tau)\beta}}\\
        &\hspace{2.75cm} e^{p\gamma\tau\langle \xi',\eta\rangle} e^{-p \gamma\tau s\frac{f(\eta)+g(\eta)}{2}} J(\eta, \beta s)^{\frac{\tau\alpha}{\tau\alpha+(1-\tau)\beta}}\frac{\dif\eta}{|K|} \\
        &= \frac1p \log \int_{\pi_{e_n^\perp}(K)} e^{p\gamma (1-\tau)\langle \xi,\eta\rangle+p\gamma\tau\langle \xi',\eta\rangle} e^{p\gamma((1-\tau) t-\tau s)\frac{f(\eta)+g(\eta)}{2}} J(\eta, \alpha t)^{\frac{(1-\tau)\beta}{\tau\alpha+ (1-\tau)\beta}} J(\eta, \beta s)^{\frac{\tau\alpha}{\tau\alpha+(1-\tau)\beta}}\frac{\dif\eta}{|K|} \\
        &=  \frac1p \log \int_{\pi_{e_n^\perp}(K)} e^{p\gamma (1-\tau)\langle \xi,\eta\rangle+p\gamma\tau\langle \xi',\eta\rangle} J(\eta, \alpha t)^{\frac{(1-\tau)\beta}{\tau\alpha+ (1-\tau)\beta}} J(\eta, \beta s)^{\frac{\tau\alpha}{\tau\alpha+(1-\tau)\beta}}\frac{\dif\eta}{|K|}
\end{aligned}
\end{equation}
because
\begin{equation}\label{gammaEq}
    \gamma= \frac{\alpha\beta}{\tau\alpha+(1-\tau)\beta}, 
\end{equation}
and
\begin{equation*}
    (1-\tau)t- \tau s= \frac{s}{t+s} t- \frac{t}{t+s}s=0. 
\end{equation*}
By the log-concavity of $J$ in the second variable (Claim \ref{sinhLogConvex} below), and \eqref{gammaEq},
\begin{equation}\label{Jineq}
    \begin{aligned}
         J(\eta, \alpha t)^{\frac{(1-\tau)\beta}{\tau\alpha+ (1-\tau)\beta}} J(\eta, \beta s)^{\frac{\tau\alpha}{\tau\alpha+(1-\tau)\beta}}&\geq J\left( \eta, \frac{(1-\tau)\alpha\beta}{\tau\alpha+(1-\tau)\beta}t + \frac{\tau \alpha \beta}{\tau\alpha+(1-\tau)\beta}s\right) \\
         &= J(\eta, (1-\tau)\gamma t+ \tau \gamma s) \\
         &= J\left( \eta, \gamma \frac{ts}{t+s}+ \gamma\frac{ts}{t+s}\right) \\
         &= J\left( \eta, \gamma\frac{2ts}{t+s}\right) \\
         &= J(\eta, \gamma r), 
    \end{aligned}
\end{equation}
because $(1-\tau)= \frac{s}{t+s}, \tau=\frac{t}{t+s}$ and $r=\frac{2ts}{t+s}$. Finally, by \eqref{weirdBigIneq}, \eqref{Jineq}, and \eqref{Eq1}, 
\begin{equation*}
    \begin{aligned}
        &\frac{(1-\tau)\beta}{\tau\alpha+ (1-\tau)\beta} h_{p,K}(\alpha(\xi,t))+ \frac{\tau\alpha}{\tau\alpha+(1-\tau)\beta}h_{p,K}(\beta(\xi',-s)) \\
        &\geq \frac1p \log \int_{\pi_{e_n^\perp}(K)} e^{p\gamma (1-\tau)\langle \xi,\eta\rangle+p\gamma\tau\langle \xi',\eta\rangle} J(\eta, \gamma r)\frac{\dif\eta}{|K|} \\
        &= h_{p,\sigma_{e_n}K} (\gamma((1-\tau)\xi+\tau\xi', r)), 
    \end{aligned}
\end{equation*}
as desired. 
\end{proof}

\begin{claim}\label{sinhLogConvex}
    For any $x>0$, $t\mapsto \log\left(\frac1t \sinh(tx)\right), t>0,$ is convex.
\end{claim}
\begin{proof}
    Write 
    $$f(t)\defeq \log\left(\frac1t \sinh(tx)\right)= \log(\sinh(tx))-\log t.
    $$ 
    Compute the derivatives, $f'(t)= x\frac{\cosh(tx)}{\sinh(tx)}-\frac1t$ and 
    \begin{equation*}
    \begin{aligned}
        f''(t)&= x^2\frac{\sinh(tx)}{\sinh(tx)}- x^2\frac{(\cosh(tx))^2}{(\sinh(tx))^2}+\frac{1}{t^2}= x^2\left(1- \frac{(\cosh(tx))^2}{(\sinh(tx))^2}+\frac{1}{(tx)^2}\right)\\
        &= x^2\left(1- \frac{1+(\sinh(tx))^2}{(\sinh(tx))^2}+ \frac{1}{(tx)^2}\right)= x^2\left(\frac{1}{(tx)^2}- \frac{1}{(\sinh(tx))^2}\right)\geq 0,
    \end{aligned}
    \end{equation*}
    because $\sinh(y)\geq y$, for all $y\geq 0$.
\end{proof}

\subsection{Finishing the proof}
\label{FinishingSection}

\begin{proof}[Proof of Proposition \ref{SteinerProp}]
Let
\begin{equation*}
    L\defeq K- s_p(\sigma_uK). 
\end{equation*}
By Lemma \ref{SantaloPointLemma}, 
$s_p(\sigma_uK)\in u^\perp$, thus,
by \eqref{perp},
$
    \sigma_u L= (\sigma_u K)- s_p(\sigma_uK). 
$
Therefore, $s_p(\sigma_u L)=0$, and hence, by \eqref{tMpDef}, 
\begin{equation}\label{tMMsigmauL}
    \M_p(\sigma_uL)=\inf_{x\in\R^n}\M_p((\sigma_uL)-x).
\end{equation}
By Lemma \ref{SeparatingLemma}, there is $t\in \R$ such that $u^\perp$ $1/2$-separates $(L-tu)^{\circ,p}$. By Lemma \ref{VolumeLemma} for $\lambda=1/2$, 
\begin{equation}\label{sigmaLeq3}
    |(\sigma_u(L-tu))^{\circ,p}|\geq |(L-tu)^{\circ,p}|. 
\end{equation}
By \eqref{parallel}, $\sigma_u(L-tu)= \sigma_uL$, and hence,
\begin{equation}\label{sigmaLeq2}
(\sigma_u(L-tu))^{\circ,p}= (\sigma_uL)^{\circ,p}.
\end{equation}
By \eqref{MpEq}, \eqref{sigmaLeq3}, and \eqref{sigmaLeq2}, 
\begin{equation}\label{LSLineq}
    \begin{aligned}
       \M_p(L-tu)&= n! |L| |(L-tu)^{\circ,p}|\leq n! |L| |(\sigma_u(L-tu))^{\circ,p}|\\
       &= n! |L||(\sigma_u L)^{\circ,p}|= \M_p(\sigma_uL),
    \end{aligned}
\end{equation}
because $|L|=|\sigma_uL|$ (Lemma \ref{SteinerProperties} below).
Note that $L-tu= K-tu-s_p(\sigma_uK)$ is a translate of $K$, thus
\begin{equation}\label{tMKML}
    \inf_{x\in\R^n}\M_p(K-x)\leq \M_p(L-tu).
\end{equation}
Combining \eqref{tMKML}, \eqref{LSLineq}, and \eqref{tMMsigmauL} proves the claim.
\end{proof}

\begin{proof}[Proof of Theorem \ref{LpSantaloThm}]
  Denote by $e_1, \ldots, e_n$ the standard basis in $\R^n$. Let $K_0\defeq K-s_p(K)$, and
  \begin{equation*}
      K_i\defeq (\sigma_{e_i}K_{i-1})- s_p(K_{i-1})
  \end{equation*}
  for $i=1, \ldots, n$. By \eqref{tMpDef} and Proposition \ref{SteinerProp}, 
  \begin{equation}\label{recursiveEq}
      \inf_{x\in \R^n} \M_p(K-x)= \M_p(K_0)\leq \M_p(K_1)\leq \ldots\leq \M_p(K_n).
  \end{equation}
  By Lemma \ref{SteinerOrthogonal} below, since $e_1, \ldots, e_n$ are all orthogonal to each other, $\sigma_{e_i}K_{i-1}$ is symmetric with respect to $e_1, \ldots, e_i$. By Lemma \ref{SantaloPointLemma}, $s_p(\sigma_{e_i}K_{i-1})\in e_1^\perp\cap\ldots\cap e_i^\perp$, thus $(\sigma_{e_i}K_{i-1})- s_p(K_{i-1})$ is also symmetric with respect to $e_1, \ldots, e_i$. Therefore, $K_n$ is symmetric with respect to all $e_1, \ldots, e_n$, and hence it is a symmetric convex body. By \eqref{recursiveEq} and Theorem \ref{SantaloSym}, $\inf_{x\in \R^n}\M_p(K)\leq \M_p(K_n)\leq \M_p(B_2^n)$.
\end{proof}

\begin{remark}  
For a short proof of 
    Theorem \ref{SantaloSym}, note that for a symmetric convex body $K$, $K^{\circ,p}$ is also symmetric \cite[Theorem 1.2]{BMR}, thus $u^\perp$ $1/2$-separates it. By Lemma \ref{VolumeLemma}, Steiner symmetrization increases its $L^p$-Mahler volume, $\M_p(K)\leq \M_p(\sigma_uK)$. Repeated applications of Steiner symmetrization converge to $B_2^n$ \cite[Theorem 9.1]{gruber}, and hence, by the continuity of $\M_p$ \cite[Lemma 5.11]{BMR}, $\M_p(K)\leq \M_p(B_2^n)$.
\end{remark}

\bigskip
 
{\sc University of Maryland}

{\tt vmastr@umd.edu}

\end{document}